\documentclass[twocolumn,english,groupedaddress, superscriptaddress,aps,pre,10pt,floatfix]{revtex4-2}
\usepackage[T1]{fontenc}
\usepackage[utf8]{inputenc}
\usepackage{amsmath,amsthm,mathtools,amssymb}
\usepackage{graphicx}
\usepackage{units}
\usepackage{dsfont}

\usepackage{url}

\usepackage[breaklinks]{hyperref}

\newtheorem{cor}{Corollary}
\newtheorem{lem}{Lemma}

\newcommand{\matr}[1]{{\boldsymbol{#1}}}
\renewcommand{\vec}[1]{{\boldsymbol{#1}}}

\newcommand{\EL}{\mathcal{M}}
\newcommand{\nn}{{N_n}}
\newcommand{\nl}{{N_e}}

\newcommand{\eye}{\mathds{1}}

\newcommand{\Bd}{\mathfrak{B}}

\begin{document}

\title{Unified algebraic deviation of distribution factors in linear power flow}

\author{Joost van Dijk}%
 \email{joost.van.dijk@tennet.eu}
  \affiliation{TenneT TSO BV, Utrechtseweg 310, 6812 AR
   Arnhem, The Netherlands}

\author{Nico Westerbeck}
  \affiliation{InstaDeep, Kemperplatz 1, 10785 Berlin, Germany}
  \affiliation{Elia Group, Heidestraße 2, 10557 Berlin, Germany}

\author{Lars Schewe}
  \affiliation{University of Edinburgh, School of Mathematics and Maxwell Institute for Mathematical Sci-
ences, James Clerk Maxwell Building, King’s Buildings, Peter Guthrie Tait Road, Edinburgh, EH9 3FD,
UK}

\author{Andrea Benigni}
  \affiliation{Forschungszentrum J\"ulich, Institute of Climate and Energy Systems -- Energy System Engineering (ICE-1), 52428 J\"ulich, Germany}
  \affiliation{RWTH Aachen University, Aachen, 52062, Germany}
  \affiliation{JARA-Energy, Jülich, 52425, Germany}

\author{Dirk Witthaut}
 \email{d.witthaut@fz-juelich.de}
  \affiliation{Forschungszentrum J\"ulich, Institute of Climate and Energy Systems -- Energy System Engineering (ICE-1), 52428 J\"ulich, Germany}
  \affiliation{Institute for Theoretical Physics, University of Cologne, K\"oln, 50937, Germany}
  \affiliation{JARA-Energy, Jülich, 52425, Germany}

\begin{abstract}
Distribution factors are indispensable tools in the design and analysis of power transmission grids. Recently, they received a renewed interest in the field of topology optimization, leading to the definition of bus merge and bus split distribution factors.
In this article, we introduce a unified derivation of the most relevant distribution factors based on matrix algebraic manipulations. This approach facilitates the generalization to more complex grid modification, in particular simultaneous switching events or bus splits. 
\end{abstract}
  
\maketitle

\section{Introduction}

Distribution factors are ubiquitous in the analysis and operation of power transmission grids and have been used for decades~\cite{landgren1972transmission}. Power transfer distribution factors (PTDFs) and line outage distribution factors (LODFs) are introduced in various text books such as~\cite{wood2013power}. They describe the change of real power flows after the change of power injections (PTDFs) or line outages (LODFs). More recently, bus merge and bus split distribution factors were introduced in~\cite{goldis2016shift,van2023bus}. 

The impact of topology changes is commonly evaluated by mapping the problem to PTDFs. For instance, a line outage can be treated by adding fictitious power injections at the terminal buses of the line. This approach is very instructive and thus used in text books such as~\cite{wood2013power}. However, it comes at the expense of a heavy notation, in particular in the case of bus split distribution factors~\cite{van2023bus}, which may impede the generalization to more complex situations. 

The ongoing transition of the electric power system entails great challenges for flexibility~\cite{entsoe_position_flex} and stability~\cite{milano2018foundations}. Grid planning and operation require novel tools to face these challenges~\cite{marot2022perspectives, viebahn24}. For instance, congestion is rapidly increasing in the European transmission grid, causing costs in the billions of Euros every year~\cite{titz2024identifying}. Dynamic topology reconfiguration may reduce congestion with existing infrastructures~\cite{subramanian2021exploring}, but generally leads to hard optimization problems. Distribution factors are widely used in planning and optimization and should be able to efficiently describe arbitrary changes in the grid topology. 

In this article we introduce a unified algebraic derivation of distribution factors, including line modifications, line outages, line closing, bus splits and bus mergers. The common idea is to formulate any modification as a low-rank matrix update and then apply the Woodbury matrix identity and related tools generalizing ideas presented in~\cite{alsac1983sparsity}.
The approach is highly flexible and thus readily generalized to multiple simultaneous modifications of the grid. We are confident that this unified approach promotes the applicability of distribution factors, in particular in dynamic topology reconfiguration, and enables further generalizations.  

The article is organized as follows. We provide a brief overview over the relevant literature in Sec.~\ref{sec:literature} and introduce the background and the necessary tools from algebraic graph theory in Sec.~\ref{sec:background}. In Sec.~\ref{sec:single}, we discuss modifications of a single branch or bus, starting from line modification and line outages and then preceding to bus merges and bus splits. We generalize these results to multiple branch modifications in Sec.~\ref{sec:multiple}.


\section{Literature overview}
\label{sec:literature}

The linear power flow or DC approximation is widely used in power system analysis and design~\cite{wood2013power}. This approximation is applicable if voltages are close to the set point and losses and grid loads are not too high. It applicability has been investigated numerically in ~\cite{purchala2005usefulness,stott2009dc} and analytically in~\cite{hartmann2024synchronized}. The application of DC power flow and various distribution factors in unit commitment models was reviewed in \cite{van2014dc}, including some topological aspects such as the treatment of phase shifting transformers.

Different algebraic methods for treating topology modifications were discussed in the literature. Sparsity is a common theme in linear power flow and can be exploited to improve computation efficiency~\cite{tinney1985sparse}. The application of sparse methods in the analysis of the topology changes was pioneered in~\cite{alsac1983sparsity,kim1985contingency}. An early review of sparse matrix methods was provided in \cite{stott1987overview}.

A purely algebraic derivations of LODFs for multiple outages was presented in~\cite{guler2007generalized} and  \cite{guo2009direct}, where the latter article focuses on the application of the Woodbury matrix identity or inverse matrix modification lemma.  Similar methods where used in~\cite{schaub2014structure,ranjan2014incremental} to address problems in general flow networks beyond power grids, including in particular the effective resistance of graphs. Graph theoretical aspects of this approach were discussed in~\cite{strake2019non,kaiser2020collective}, including the scaling of LODFs with distance and the analysis of collective effects of multiple outages. Line closings have been treated in~\cite{sauer2001extended}. Extensions of the Woodbury matrix identity to pseudo inverses relevant for network flow problems were treated in~\cite{meyer1973generalized}. The case of network splits under multiple outages was treated in~\cite{zocca2021spectral}. An alternative approach to topology changes based on the superposition approach was presented~\cite{marot2023extended}. 

A graph dual approach to PTDFs and LODFs has been introduced in~\cite{ronellenfitsch2016dual,ronellenfitsch2017dual}. This approach shifts the attention from buses to branches and carries out all calculations in terms of real power flows instead of phase angles. This formulation is explicitly useful in linear optimal power flow problems~\cite{carvalho1988optimal,horsch2018linear}.

Distribution factors are formulated in terms of the inverse of the nodal susceptance matrix. Topology modifications are then described by low-rank updates of the inverse as described above. Alternatively, one can work on the level of the linear set of equations that constitute the DC approximation. In practice, such linear equations are solved via an LU decomposition of the nodal susceptance matrix~\cite{tinney1967direct}. Efficient formula to update this LU decomposition after a low-rank change of the matrix $\matr B$ are available~\cite{bennett1965triangular,suhl1993fast,irisarri1981automatic}.
Partial matrix factorizations to update for outages  and other topology changes were treated in~\cite{chan1986partial}.

\section{Background and Notation}
\label{sec:background}

In this section we summarize the fundamental equations and introduce some useful notation. Linear power flow describes the flow of real power in high voltage grids, neglecting Ohmic losses as well as changes in the voltage magnitude~\cite{wood2013power,witthaut2022collective}. The power flow equations are linearized, such that the flow depends linearly on the difference of the nodal voltage phase angles. 

So let $\theta_n$ denote the voltage phase angle at a bus or node $n$. The real power flow over a branch or edge $(n,k)$ reads reads
\begin{align}
    f_{n \rightarrow k} = b_{nk} (\theta_n  - \theta_k),
    \label{eq:flow-single-branch}
\end{align}
where $b_{nk}$, is the effective susceptance of the branch $(n,k)$. The flows have to satisfy the energy conservation law or Kirchhoff current law at every bus or node, 
\begin{align}
    p_n = \sum_{k} f_{n \rightarrow k} .
    \label{eq:kcl-single-node}
\end{align}
where $p_n$ is the power injection as bus $n$, i.e. the difference of power generation and consumption including the consumption by shunt elements. We will always assume that the grid is balanced, i.e. that $\sum_n p_n = 0$. The two equations \eqref{eq:flow-single-branch} and \eqref{eq:kcl-single-node}
fully determine the real power flow in the grid. The validity of the linear power flow or DC approximation was studied in detail in ~\cite{purchala2005usefulness,stott2009dc,hartmann2024synchronized}.

We introduce a compact vectorial notarial that will facilitate the further analysis. Buses or nodes are labeled consecutively as $n = 1,\ldots,\nn$ and all power injections are summarized in a vector $\vec p = (p_1, \ldots , p_\nn)^\top$. Branches or edges are either labeled by their endpoints $(n,m)$ or consecutively as $e = 1,\ldots,\nl$. To keep track of the direction of power flows we fix an orientation for each edge.  
The grid topology and the orientation is summarized in the node-edge incidence matrix $\matr E \in \mathbb{R}^{\nn \times \nl}$ with components \cite{Newm10}
\begin{equation}
   E_{n,e} = \left\{
   \begin{array}{r l}
      + 1 & \; \mbox{if line $e$ starts at node $n$},  \\
      - 1 & \; \mbox{if line $e$ ends at node $n$},  \\
      0     & \; \mbox{otherwise}.
  \end{array} \right.
  \label{eq:def-nodeedge}
\end{equation}
The real power flows are summarized in a vector $\vec f = (f_1,\ldots, f_\nl)^\top$ and the branch susceptances in the diagonal matrix $\matr \Bd = \mbox{diag}(b_1,\ldots,b_\nl)$.


The two defining equations of the linear power flow equations are then written in vectorial form as
\begin{align}
    \vec f &=  \matr \Bd  \matr E^\top \vec \theta, \\
    \vec p &= \matr E \vec f.
\end{align}
Substituting the first equation into the second one thus yields a linear equation connecting power injections and voltage phase angles,
\begin{align}
    \vec p = \matr E \matr \Bd  \matr E^\top \vec \theta  = \matr B \vec \theta \, .
    \label{eq:DC-pBt}
\end{align}
The nodal susceptance matrix $\matr B \in \mathbb{R}^{\nn \times \nn}$ can we written in components as
\begin{equation}
   B_{n,k} = \left\{
   \begin{array}{r l}
         \sum_{l=1}^\nn b_{nl}  & \; \mbox{for} \, n = k, \\
         - b_{nk} & \; \mbox{if $n\neq m$ and $n$, $m$ are adjacent},  \\
      0     & \; \mbox{otherwise}.
  \end{array} \right.
  \label{eq:def-Laplacian}
\end{equation}
Technically, $\matr B$ is a Laplacian matrix \cite{Newm10}. It always has one zero eigenvalue with the corresponding eigenvector $(1,1,\ldots,1)^\top$ such that it is not invertible. This is no major problem, the set of equations \eqref{eq:DC-pBt} is solvable since we  assume that the grid is balanced $\sum_n p_n=0$. In practice, we can use the Moore-Penrose pseudo-inverse $\matr B^{+}$. In the current setup, the pseudo-inverse can be calculated as \cite{ranjan2014incremental}
\begin{equation}
  \matr{B}^+ = \left( \matr{B} + \frac{1}{N_n} \matr{J}\right)^{-1} 
      -\frac{1}{N_n} \matr{J},
\end{equation}
where $\matr{J} \in \mathbb{R}^{N_n \times N_n}$ is a matrix of all 1s.

One can avoid using the pseudo-inverse by fixing a slack bus $k$ where the voltage is fixed as $\theta_k=0$. We then remove the $k$th row of the power injection vector $\vec p$, the phase angle vector $\theta$ and the node-edge incidence matrix $\matr E$. The nodal susceptance matrix $\matr B = \matr E \matr \Bd \matr E^\top$ then has dimension $(N_n-1) \times (N_n-1)$ and is of full rank if the grid is connected. In the mathematical literature, $\matr B$ is referred to as a grounded Laplacian. We can thus work with the ordinary inverse $\matr B^{-1}$ instead of the pseudo-inverse. In the remainder of this article we will follow this definition for the sake of simplicity.

We emphasize that the inverse of the grounded Laplacian $\matr B^{-1}$ or equivalently the pseudo-inverse of the Laplacian encodes the essential aspects of the power grid as it relates the state variables to the power injections, 
\begin{align*}
    \vec \theta = \matr B^{-1} \vec p. 
\end{align*}
Hence, our analysis will mainly focus on the inverse $\matr B^{-1}$. We can further augment this matrix to describe the relation of flows and power injections and flows,
\begin{align}
    \vec f = \underbrace{\matr \Bd \matr E^\top \matr B^{-1}}_{=: \matr{PTDF} } \vec p.
    \label{eq:def-ptdf}
\end{align}
The matrix $\matr{PTDF}$ is referred to as the matrix of power transfer distribution factors.

A summary of the symbols and variables used in this article is provided in Table~\ref{tab:notation} to improve the readability.

\begin{table}[tb]
\caption{
List of symbols and  variables. Vectors are written as boldface lowercase  letters and matrices as boldface uppercase case letters. 
\label{tab:notation}
}
\begin{tabular}{p{1cm} p{7cm}}
  \hline
  $b_{e}$ & effective susceptance of branch $e$\\
  $\Delta b_{e}$ & change of effective susceptance of branch $e$\\
  $\matr \Bd$ & diagonal matrix of all branch susceptances\\
  $\matr B_x$ & Nodal susceptance matrix for the grid topology $x$\\
  $e$, $\ell$ & indexing variables typically used for branches\\
  $\matr E$ & node-edge incidence matrix \\
  $f_e$ & real-power flows over branch $e$\\
  $\vec f$ & vector of all real-power flows\\
  $\matr J$ & matrix of all 1s\\
  $n,k,l$ &  indexing variables typically used for buses\\
  $\nn$ & number of buses or nodes in a grid\\
  $\nl$ & number of branches or edges in a grid\\
  $\vec \nu_e$ & vector with $\pm 1$ at terminal ends of branch $e$\\
  $p_n$ & real power injection at bus $n$\\
  $\vec p$ & vector summarizing all real power injections\\
  $\theta_n$ & voltage phase angle at bus $n$\\
  $\vec \theta$ & vector of all voltage phase angles\\
  $\eye$ & unit matrix \\
  $\vec u_n$ & $n$th standard unit vector\\
  ${}^\top$ & superscript denoting matrix transpose\\
  \hline
\end{tabular}
\end{table}

\section{Single branch modifications}
\label{sec:single}

\subsection{Branch modifications and line outages}

In this section we review basic methods to deal with finite changes at a single branch of a grid such as a line outage. While the resulting equations are standard tools in power engineering, we focus on a purely algebraic deviation which facilitates the generalization in the following sections.

So assume that the initial reference grid is described by the grounded Laplacian $\matr B_r$. Then the grid is subject to a modification of an arbitrary branch $e$ such that the effective susceptance changes as 
\begin{align}
    b_e \rightarrow \tilde b_e = b_e + \Delta b_e.
\end{align}
In case of a line outage we have $\tilde b_e = 0$ or, equivalently, $\Delta b_e = - b_e$.

This entails a change of the matrix $\Bd$ and the grounded  Laplacian matrix
\begin{align*}
    \matr \Bd_r &\rightarrow \matr \Bd_m = \matr \Bd_r + \Delta b_e \vec u_e \vec u_e^\top. \\
    \matr B_r &\rightarrow \matr B_m  = \matr B_r + \Delta  b_e \vec \nu_e \vec \nu_e^\top ,
\end{align*}
where we use the subscripts $r$ and $m$ to denote the reference and the modified topology, respectively. Here and in the following, $\vec u_e$ denotes the $e$th standard basis vector in the space of branches and $\vec \nu_e = \matr E \vec u_e$. Physically, the vector $\vec \nu_e$ describes a unit power injection at the from end of branch $e$ and a unit power ejection at the to end of branch $e$.

We evaluate how this topology change affects state variables and power flows assuming that the power injection vector $\vec p$ remains unchanged. To this end, we compute the inverse of the grounded Laplacian using the Sherman–Morrison formula~\cite{hager1989updating}.
\begin{align}
    \matr B_m^{-1} &= (\matr B_r+ \Delta b_e \vec \nu_e \vec \nu_e^\top)^{-1}  
    \nonumber \\
    &= \matr B_r^{-1}   -  \matr B_r^{-1} \vec \nu_e  \Delta b_e 
    (1+\Delta b_e  \vec \nu_e^\top  \matr B_r^{-1}  \vec \nu_e )^{-1} 
    \vec \nu_e^\top \matr B_r^{-1} 
    \label{eq:sm-single-finite}
\end{align}
We can now directly read of the PTDF matrix for the new grid topology,
\begin{align}
    \label{eq:ptdf-single-mod} 
    \matr{PTDF}_m &= \matr \Bd_m \matr E^\top \matr B_m^{-1} \\
    &= \left[ \eye +  \frac{\Delta b_e}{b_e} \vec u_e \vec u_e^\top \right]
    \matr{PTDF}_r 
    \nonumber \\
    & \quad  \times \left[ \eye - 
    \frac{\Delta b_e}{1+\Delta b_e  \vec \nu_e^\top  \matr B_r^{-1}  \vec \nu_e}
    \vec \nu_e  \vec \nu_e^\top \matr B_r^{-1} \right]. \nonumber
\end{align}
Furthermore, we can also discuss the impacts of the topology modification on the level of the real power flows. We denote the flows in the reference grid topology as $\vec f_r$ and the flows in the modified topology as $\vec f_m$. Using equation \eqref{eq:def-ptdf} and $\vec p = \matr E \vec f_r$, we obtain
\begin{align*}
    \vec f_m &= \matr \Bd_m \matr E^\top \matr B_m^{-1} \matr E \vec f_r  \\
    &= \matr \Bd_m \matr E^\top \matr B_r^{-1} \matr E \vec f_r \\
    & \quad - \frac{\Delta b_e} 
    {1+\Delta b_e  \vec \nu_e^\top  \matr B_r^{-1}  \vec \nu_e}
    \matr \Bd_m \matr E^\top \matr B_r^{-1} \vec \nu_e  
    \vec \nu_e^\top \matr B_r^{-1} \matr E \vec f_r  \\
    &= \left( \eye + \frac{\Delta b_e}{b_e} \vec u_e \vec u_e^\top \right) \vec f_r \\
    & \quad - \frac{\Delta b_e/b_e} 
    {1+\Delta b_e  \vec \nu_e^\top  \matr B_r^{-1}  \vec \nu_e}
    \matr \Bd_m \matr E^\top \matr B_r^{-1} \vec \nu_e  
    \vec u_e^\top \vec f_r
\end{align*}
In the last step, we have used the relations
\begin{align*}
     \matr \Bd_m &= \left( \eye + \frac{\Delta b_e}{b_e} \vec u_e \vec u_e^\top \right) \matr \Bd_r\\
    \matr \Bd_r \matr E^\top \matr B_r^{-1} \matr E \vec f_r &= \vec f_r \\
    \vec u_e^\top \matr E^\top \matr B_r^{-1} \matr E \vec f_r  
    &= b_e^{-1} \vec u_e^\top  \matr \Bd_r \matr E^\top \matr B_r^{-1} \matr E \vec f_r
    = b_e^{-1} \vec u_e^\top \vec f_r.
\end{align*}
The changes of the real power flows are particularly interesting in the case of a line outage, hence we now set $\Delta b_e = - b_e$. Then the previous equation can be rewritten as
\begin{align}
    \vec f_m &= 
     \vec f_r + \underbrace{
     \left( \frac{1}{\scriptstyle 1- b_e  \vec \nu_e^\top  \matr B_r^{-1}  \vec \nu_e} 
            \matr \Bd_m  \matr E^\top \matr B_r^{-1}   \matr E 
     - \eye  \right) \vec u_e 
     }_{=:\matr{LODF}_{:,e}}
     \vec u_e^\top \vec f_r
     \label{eq:lodf-def}
\end{align}
Here, $\matr{LODF}_{:,e}$ denotes the $e$th column of the matrix of line outage distribution factors. 

To conclude the treatment of line outages we rewrite the results in terms of the power transfer distribution factors to enable a direct comparison to the literature. The PTDFs of the references grid are written as
\begin{align}
    PTDF^{(r)}_{e,i} &= \vec u_e^\top \matr{PTDF}_r  \, \vec u_i
    = b_e \vec \nu_e^\top \matr{B}_r^{-1} \vec u_i \, . \label{eq:ptdf_ref}
\end{align}
Furthermore, we denote by $\matr{PTDF}^{(r)}_{e,:}$ the $e$th row and $\matr{PTDF}^{(r)}_{:,i}$ the $i$th column of the PTDF matrix. For a branch $e = (i,j)$, Eq.~\eqref{eq:ptdf-single-mod} then assumes the familiar form
\begin{align*}
    &\matr{PTDF}_m = \matr{PTDF}_r \\
    & \qquad - \frac{\matr{PTDF}^{(r)}_{:,i} - \matr{PTDF}^{(r)}_{:,j}}{   
    \frac{b_e}{\Delta b_e } + ({PTDF}^{(r)}_{e,i} - {PTDF}^{(r)}_{e,j}) }
    \matr{PTDF}^{(r)}_{e,:} \, .
\end{align*}
for all rows except for the $e$th row that vanishes completely. The power flow in the modified grid topology then read
\begin{align*}
    \vec f_m = \vec f_r 
    - \frac{\matr{PTDF}^{(r)}_{:,i} - \matr{PTDF}^{(r)}_{:,j}}{   
    \frac{b_e}{\Delta b_e } + ({PTDF}^{(r)}_{e,i} - {PTDF}^{(r)}_{e,j}) }
    \vec u_e^\top \vec f_r \, .
\end{align*}
Notably, $[b_e + \Delta b_e({PTDF}^{(r)}_{e,i} - {PTDF}^{(r)}_{e,j}) ]^{-1} \vec u_e^\top \vec f_r$
tells the angle difference between $i$ and $j$ after the outage.

Finally, we can treat the closing of a transmission line $e$ in a similar way. In this case we have $b_e=0$, but we have to ensure that the line $e$ is included in the incidence matrix $\matr E$ and the diagonal matrix $\matr \Bd$. Then we have 
\begin{align}
    \vec f_m &= 
     \vec f_r + \matr{LCDF}_{:,e} \vec u_e^\top \vec \theta_r
\end{align}
where the line closing distribution factors (LCDF) are defined as
\begin{align}
     \matr{LCDF}_{:,e} = 
     \Delta b_e
     \left( \frac{1}{\scriptstyle 1 + \Delta b_e  \vec \nu_e^\top  \matr B_r^{-1}  \vec \nu_e} 
            \matr \Bd_m  \matr E^\top \matr B_r^{-1}   \matr E 
     - \eye  \right) \vec u_e 
     \label{eq:lcdf-def}
\end{align}
The PTDF matrix in the modified grid topology then reads
\begin{align*}
    &\matr{PTDF}_m = \matr{PTDF}_r \\
    & \qquad - (\matr{PTDF}^{(r)}_{:,i} - \matr{PTDF}^{(r)}_{:,j})\frac{  \vec \nu_e^\top \matr B_r^{-1}}{\frac{1}{\Delta b_e} + \vec \nu_e^\top  \matr B_r^{-1}  \vec \nu_e}
     \, ,
\end{align*}
for all rows except for the $e$th row that requires a more careful treatment.

\subsection{Phase shifting transformers}

Phase-shifting transformers (PSTs) are commonly used to optimize active power flows by setting the phase shift angle~\cite{verboomen2005phase}. In the DC approximation, PSTs can be included in two equivalent ways. First, the effect of a phase shifting transformer $(k,l)$ can be described by effective power injections injections at the terminal buses $k$ and $l$. Let $b_{kl}$ denote the series susceptance of the transformer and $\theta_{kl}^{\rm shift}$ the phase shifting angle. Then the effective power injections are given by 
\begin{align}
    \vec{\hat p} = \vec p - b_{kl} \theta_{kl}^{\rm shift} 
    \, \vec \nu_{kl}.
    \label{eq:pst-injection}
\end{align}
The flow on all branches except for the PST itself are then given by
\begin{align*}
    \vec f = \matr \Bd \matr E^\top \matr B^{-1}
    \vec{\hat p} = 
    \matr{PTDF} \, \vec{\hat p},
\end{align*}
while a correction has to be applied for the branch $(k,l)$ itself (see below). A modification of the PST can be easily combined with any other topology modification. A change of the PST phase angle is described by a change of the effective injection $\vec{\hat p}$, while the topology modification is accounted for by a change of the PTDF matrix.

Alternatively, we can describe the impact of PSTs in terms of Phase Shifter Distribution Factors (PSDF)~\cite{van2014dc}. In this case one does \emph{not} include the PSTs in the effective power injections. We collect all phase shifting angles in a vector $\vec{\theta}^{\rm shift} = (\theta_1^{\rm shift}, \ldots,  \theta_{N_e}^{\rm shift})^\top$, where we set $\theta_e^{\rm shift} = 0$ for every ordinary transmission line $e$. Then the real power flows are given by
\begin{align}
    \vec f &= \matr{PTDF} \, \vec{\hat p} \nonumber \\
    &= \matr{PTDF} \,  \vec{p} + 
             \matr{PSDF} \, \vec{\theta}^{\rm shift}.
\end{align}
with the PSDF matrix
\begin{align*}
    &PSDF_{\ell,e}  \\
    &\quad = \left\{
    \begin{array}{l l}
    b_e - b_e ( PTDF_{{\rm from}(e)} - PTDF_{{\rm to}(e)})
    \;
    & {\rm if} \; \ell=e \\ 
    - b_e ( PTDF_{{\rm from}(e)} - PTDF_{{\rm to}(e)})
    & {\rm if} \; \ell \neq{} e. 
    \end{array} 
    \right.
\end{align*}
The PSDF approach provides a separation between the effects due to the nodal power and the effects due to the PST, which can be useful in optimization problems.
It allows to readily integrate modifications of the PSTs and the grid topology: Changes to the PST phase angle are represented explicitly, while change of the topology are included in the PTDF and the PSDF matrices. These matrices have to be updated according to Eq.~\eqref{eq:ptdf-single-mod} in case of a branch modification.

\subsection{Bus Merge Distribution Factors}

Our previous results are readily extended to describe the closing of a switch or the merging of two busbars. An ideal switch $s$ is is treated as a branch with two possible values of the susceptance: $b_s = 0$ in the open state and $b_s \rightarrow \infty$ in the closed state. We note two important facts before we proceeds. First, the nodal susceptance matrix $\matr B$ diverges in the closed state, but the inverse $\matr B^{-1}$ is well defined as we show below. Second, modeling the closing of a busbar is far easier that modeling its opening. Here we start with the closing and come back to the opening in the following section.

To describe the closing of a switch or the merging of two busbars we start from Eq.~\eqref{eq:sm-single-finite} and carry out the limit $\Delta b_s \rightarrow \infty$, 
\begin{align}
    \matr B_m^{-1} &= 
    \matr B_r^{-1}  \Big[ \eye  -  (\vec \nu_s^\top  \matr B_r^{-1} \vec \nu_s)^{-1}
    \vec \nu_s  \vec \nu_s^\top \matr B_r^{-1} \Big].
    \label{eq:sm-single-finite2}
\end{align}
We can further invoke the limit in Eq.~\eqref{eq:ptdf-single-mod} for the power transfer distribution factors for all branches $e \neq s$,
\begin{align*}
    \widetilde{\matr{PTDF}}_m  = 
    \widetilde{\matr{PTDF}}_r  \left[ \eye + 
    ( \vec \nu_e^\top  \matr B_r^{-1}  \vec \nu_e )^{-1}
    \vec \nu_e  \vec \nu_e^\top \matr B_r^{-1} \right].
    \label{eq:ptdf-single-mod2}
\end{align*}
Here, we use the tilde to denote an incomplete PTDF matrix missing the row $s$. To determine the flow over the closed switch or busbar coupler, we have to proceed differently. Using Kichhoff's current law for the from side of the switch $s$ we find
\begin{align}
    f_s = p_{{\rm from}(s)} - \matr E_{{\rm from}(s),:} \cdot \widetilde{\matr{PTDF}}_m \cdot \vec p .
\end{align}

\subsection{Bus Split Distribution Factors}
\label{sec:bsdf}

\begin{figure*}[tb]
\includegraphics[width=\textwidth,clip,trim=0cm 1.5cm 1cm 0cm]{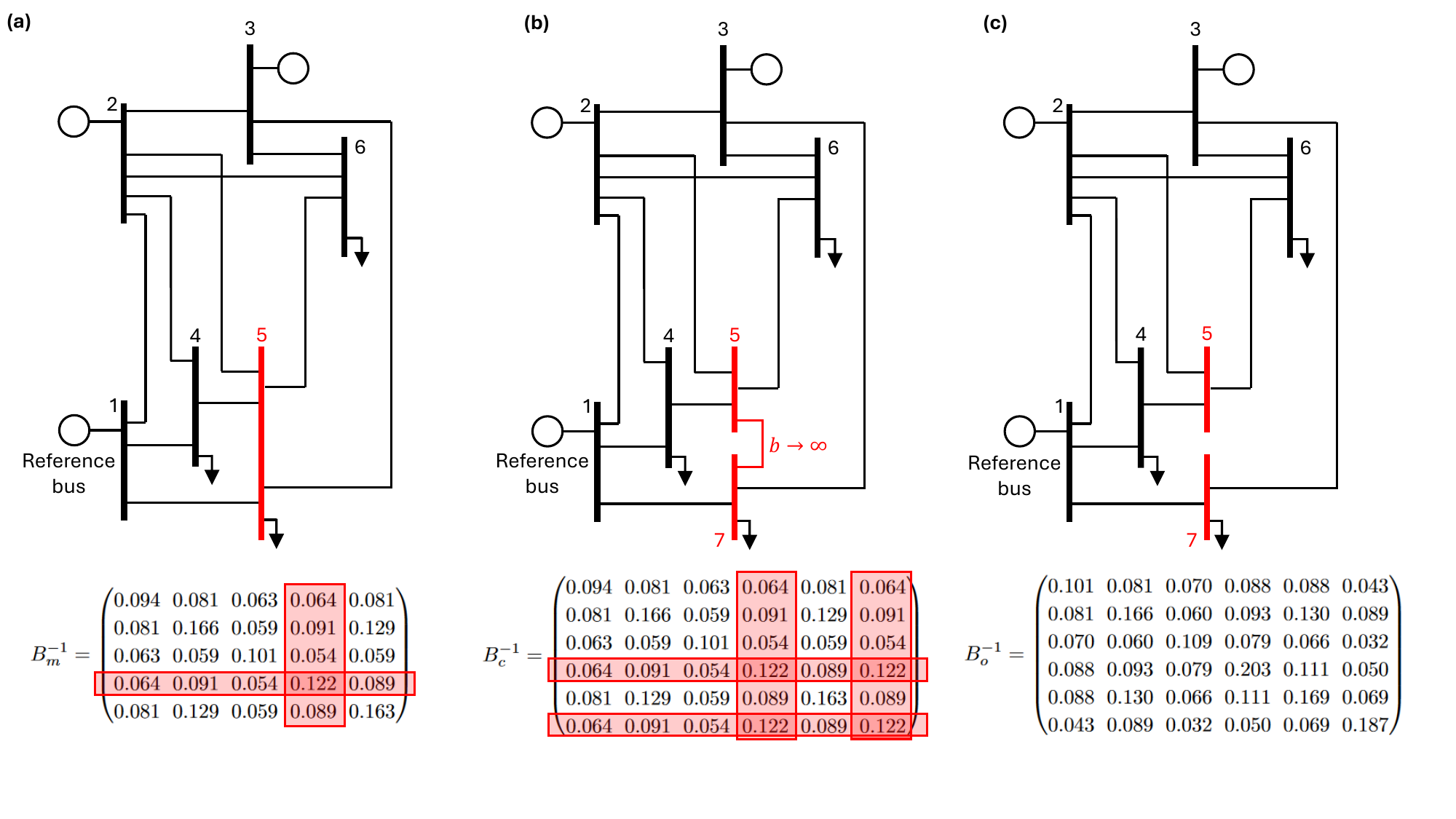}
\caption{
\label{fig:busbar-example}
A bus split in an elementary test grid. Bus number 5 is split into the two buses 5 and 7 (marked in red color).
The three columns show the different grid configurations used in the derivation of the bus split distribution factors:
(a) the `merged' configuration before the split,
(b) the `closed' configuration before the split, and
(c) the `open' configuration after the split.
The bottom row shows the inverse of the grounded Laplacian in the three different configurations $\matr B_{m,c,o}^{-1}$. 
We use the test grid \texttt{case6ww} from the software \texttt{Matpower} 8.0 \cite{zimmerman2010matpower} which is based on an example in \cite{wood2013power}.
}
\end{figure*}

In practice, bus splits are often more important than bus mergers. Consider for example the operation of busbar couplers, which are typically closed. However, it can be advisable in some cases to open one or few busbar couplers to influence the real power flow and eventually mitigate congestion. An elementary example is shown in Fig.~\ref{fig:busbar-example}. In the original grid configuration the maximum real power flow is $44.922$ pu on the branch $(3,6)$. After the busbar 5 has been split, the maximum real power flow is $42.233$ pu on the branch $(1,7)$.

In this section we introduce the algebraic treatment for a single busbar coupler. Initially, the coupler is closed such that the two busbars effectively constitute a single node of the network. Then we want to describe the impact of the opening of the coupler making the two busbars two separate nodes with no direct connection. Three technical challenges have to be dealt with. First, a switch or busbar coupler $s$ can be modeled as a branch with two possible values: $b_s \rightarrow 0$ (open) and $b_s \rightarrow \infty$ (closed). We must formulate the theory in a way that allows to evaluate the two limits in a convenient way. Second, two coupled busbars effectively constitute a single node of the network. If we start our computation from this single-node description, then we have to add a new node to describe the opening of the busbar. Finally, transmission grids must be operated in a $(n-1)$ secure way \cite{wood2013power}. Hence, we have to formulate the theory such that it allows for an efficient $(n-1)$ contingency analysis. 

Let the matrix $\matr B_m$ describe the grid where all switches or busbar couplers are closed and have been merged into a single node each. This configuration will serve as the reference grid. Furthermore, let $\matr B_o$ describe the grid where the busbar coupler has been opened. Hence, we have to describe the two uncoupled busbars as separate nodes such that the dimension of $\matr B_o$ is larger than the dimension of $\matr B_m$. This is the grid of interest which shall be analyzed in terms of congestion or $(n-1)$ stability. 

To connect the case of interest to the reference case, we introduce a third network configuration $\matr B_c$ where the switch or busbar couplers is closed, but where the nodes have \emph{not} been merged. In this configuration, the incident branches must be assigned to the node (busbar), where they are connected when the coupler opens. The two matrices $\matr B_c$ and $\matr B_o$ have the same dimension. The different configurations are shown for an elementary example in Fig.~\ref{fig:busbar-example}. 

Notably, the matrix $\matr B_c$ diverges in the limit $b_s \rightarrow \infty$, hence we have to replace it in the calculations to obtain a convergent final result. The inverse $\matr B_c^{-1}$ exists in the limit $b_s \rightarrow \infty$ and is singular. 


Our computation is based on the inverse matrices $\matr B^{-1}_{m,c,o}$. Once these matrices are known, we can easily compute the nodal phase angles $\vec \theta$ as well as the power transfer and the line outage distribution factors.  We assume that we explicitly compute $\matr B^{-1}_{m}$ once for the reference grid. Then we want to obtain $\matr B^{-1}_{o}$ for the grid of interested in an efficient way.  

We proceed via the intermediate grid configuration. The inverse matrix $\matr B^{-1}_{c}$ is obtained from $\matr B^{-1}_{m}$ by a simple padding procedure. That is, the rows and columns of the matrix $\matr B^{-1}_{m}$ corresponding to the coupled busbars are simply copied and inserted at the correct position. The only challenge for practical implementations is a correct bookkeeping. At this step, one also has to update the vectors $\vec p$ in a consistent way. Power injections must be assigned to the node where they shall be connected after the opening of the coupler.

Now we address how to compute $\matr B^{-1}_{o}$ from $\matr B^{-1}_{c}$. To this end we first assume that the busbar coupler $s$ is an ordinary branch with finite effective susceptance $b_s$. At this  stage the opening of the coupler is mathematially equivalent to a line outage. We then have to invoke the limit $b_s \rightarrow \infty$ to describe an ideal busbar coupler. To carry out this limit, we must modify the description of line outages.

The starting point is Eq.~\eqref{eq:sm-single-finite} with $\Delta b_s  = - b_s$ which yields
\begin{align}
    \matr B_o^{-1} &= \matr B_c^{-1}  +    
    b_s \matr B_c^{-1} \vec \nu_s  
    ( b_s - b_s^2  \vec \nu_s^\top  \matr B_c^{-1} \vec \nu_s )^{-1} 
    b_s \vec \nu_s^\top \matr B_c^{-1} .
    \label{eq:1coupler-finite}
\end{align}
We now have to invoke the limit $b_s \rightarrow \infty$ to model an ideal coupler, which proves difficult in the current formulation. For instance, $\matr B_c^{-1} \vec \nu_s$ vanishes while $b_s$ diverges, such that the product $b_s \matr B_c^{-1} \vec \nu_s$ is undefined in the limit.

We thus have to rewrite Eq.~\eqref{eq:1coupler-finite} in a more convenient form. To this end we derive the following matrix identities starting from the definition $\matr B_c - \matr B_o = b_s \vec \nu_s \vec \nu_s^\top$,
\begin{align}
    b_s \vec \nu_s^\top \matr B_c^{-1} 
    &= 
    \frac{1}{\vec \nu_s^\top \vec \nu_s} \vec \nu_s^\top b_s \vec \nu_s \vec \nu_s^\top \matr B_c^{-1} \nonumber \\
    &= 
   \frac{1}{\vec \nu_s^\top \vec \nu_s}  \vec \nu_s^\top \big[ \matr B_c - \matr B_o  \big] \matr B_c^{-1}  \nonumber \\
    &= 
    \frac{1}{\vec \nu_s^\top \vec \nu_s}  \vec \nu_s^\top \big[ \eye - \matr B_o \matr B_c^{-1} \big] .
    \label{eq:techlem1}
\end{align}
applying this identity repeatedly, we furthermore obtain
\begin{align}
    & b_s - b_s^2 \vec \nu_s^\top \matr B_c^{-1} \vec \nu_s 
    = 
    b_s - \frac{b_s}{\vec \nu_s^\top \vec \nu_s} \vec \nu_s^\top \big[ \eye - \matr B_o \matr B_c^{-1} \big] \vec \nu_s \nonumber \\
    & \qquad =
    \frac{b_s}{\vec \nu_s^\top \vec \nu_s}  \vec \nu_s^\top  \matr B_o  \matr B_c^{-1} \vec \nu_s 
    \nonumber \\ 
    & \qquad =
     \frac{1}{(\vec \nu_s^\top \vec \nu_s)^2} \vec \nu_s^\top \big[ \matr B_o- \matr B_o \matr B_c^{-1} \matr B_o \big]  \vec \nu_s \, .
     \label{eq:techlem2}
\end{align}
We see that the right-hand of the equations in Eq.~\eqref{eq:techlem1} and Eq.~\eqref{eq:techlem2} no longer contains the quantities $b_s$ and $\matr B_c^{-1} \vec \nu_s$ that diverge or tend to zero in the limit $b_s \rightarrow \infty$. 
The matrices $\matr B_o$ and $\matr B_c^{-1}$ are well defined such that the limit can be carried out without major problems. Hence we obtain the following results for the matrix $\matr B_o^{-1}$ describing the grid with the open coupler as well as the power transfer and line outage distribution factors.

Hence we rewrite Eq.~\ref{eq:1coupler-finite} for the matrix $\matr B_o^{-1} $ using Eq.~\eqref{eq:techlem1} and its transpose as well as Eq.~\eqref{eq:techlem2}.

\begin{lem}
The inverse grounded Laplacian of the grid with open coupler is given by
\begin{align}
    &  \matr B_o^{-1}  =  \matr B_c^{-1} 
    + \Big[  \vec \nu_s^\top \left( \matr B_o - \matr B_o \matr B_c^{-1} \matr B_o  \right) \vec \nu_s \Big]^{-1} 
    \nonumber \\
    & \quad \times  \left( \eye -  \matr B_c^{-1} \matr B_o \right) \vec \nu_s
    \vec \nu_s^\top \left( \eye - \matr B_o \matr B_c^{-1} \right) .
\end{align}
\label{cor:Bo-inverse-elements}
\end{lem}

\begin{proof}
The result follows from Eq.~\eqref{eq:1coupler-finite} using Eq.~\eqref{eq:techlem1} and its transpose as well as Eq.~\eqref{eq:techlem2}. As the resulting expression does not contain $b_s$ explicitly, the limit $b_s \rightarrow \infty$ is straightforward.
\end{proof}

\begin{cor}
The PTDF matrix of the grid with the open busbar coupler is given by
\begin{align}
    \matr{PTDF}_o = \matr{PTDF}_c +
    \vec{bsdf}  \;
       \vec \nu_s^\top \left( \eye -  \matr B_o \matr B_c^{-1} \right) 
\end{align}
with the bus split distribution factor vector
\begin{align}
    \vec{bsdf} =& 
    \Big[  \vec \nu_s^\top \left( \matr B_o - \matr B_o \matr B_c^{-1} \matr B_o \right)  \vec \nu_s 
    \Big]^{-1}  \nonumber \\
    & \qquad \times \matr \Bd_o \matr E^\top
     \left( \eye -  \matr B_c^{-1} \matr B_o   \right) \vec \nu_s \, .
\end{align}
\end{cor}


\begin{proof}
We start from Eq.~\eqref{eq:def-ptdf} for the grid with the open busbar coupler,
\begin{align*}
    \vec f_o = \underbrace{\matr \Bd_o \matr E^\top \matr B_o^{-1}}_{=: \matr{PTDF}_o } \vec p.
\end{align*}
Substituting the expression for $\matr B_o^{-1}$ from lemma \ref{cor:Bo-inverse-elements} the yields the result.
\end{proof}

\begin{cor}
The columns of the LODF matrix of the grid with the open busbar coupler are given by 
\begin{align}
     \matr{LODF}_{:,e} = 
      \left( \frac{1}{\scriptstyle 1- b_e  \vec \nu_e^\top  \matr B_o^{-1}  \vec \nu_e} 
            \matr \Bd_o  \matr E^\top \matr B_o^{-1}   \matr E 
     - \eye  \right) \vec u_e \vec u_e^\top
\end{align}
with
\begin{align*}
    & \matr E^\top \matr B_o^{-1} \matr E = \matr E^\top \matr B_c^{-1} \matr E 
    + \Big[  \vec \nu_s^\top \left(  \matr B_o - \matr B_o \matr B_c^{-1}  \matr B_o \right) \vec \nu_s \Big]^{-1} 
    \nonumber \\
    & \quad \times \matr E^\top \left( \eye -  \matr B_c^{-1} \matr B_o \right) \vec \nu_s
    \vec \nu_s^\top \left( \eye - \matr B_o \matr B_c^{-1} \right) \matr E. \\
    & \vec \nu_e^\top \matr B_o^{-1} \vec \nu_e = \vec \nu_e^\top \matr B_c^{-1} \vec \nu_e
    + \Big[  \vec \nu_s^\top \left( \matr B_o - \matr B_o \matr B_c^{-1} \matr B_o \right)  \vec \nu_s \Big]^{-1} 
    \nonumber \\
    & \quad \times \vec \nu_e^\top \left( \eye -  \matr B_c^{-1} \matr B_o \right) \vec \nu_s
    \vec \nu_s^\top \left( \eye - \matr B_o \matr B_c^{-1} \right) \vec \nu_e.
\end{align*}
\end{cor}

\begin{proof}
This result follows from the definition of the line outage distribution factors in Eq.~\eqref{eq:lodf-def}, replacing the reference grid $\matr B_r^{-1}$ by the grid with the open busbar coupler $\matr B_o^{-1}$ using lemma \ref{cor:Bo-inverse-elements}.
\end{proof}

\subsection{Alternative treatment of bus splits}

The effect of a bus split can also be computed in a different way using ideas of~\cite{kim1985contingency}. In this approach one uses another intermediate grid $\matr B_c$ topology with an additional idle bus. That is, we add a bus that is not connected to any other bus and has a power injection of zero. The bus split then corresponds to rewiring branches from the existing bus to the new one. These topology changes can be treated equivalently to line outages or closings. 

Let us illustrate this approach using the same example as in Fig.~\ref{fig:busbar-example}. The rewiring is described by the change of the grounded Laplacian
\begin{align*}
    & \Delta \matr B = \matr B_o - \matr B_c \\
    &= \begin{pmatrix}
        0          & 0 & 0          & b_{2,5} & 0 & -b_{2,5} \\
        0          & 0 & 0          & b_{3,5} & 0 & -b_{3,5} \\
        0          & 0 & 0          & 0 & 0 & 0 \\
         b_{2,5} & b_{3,5} & 0 & -b_{2,5} - b_{3,5} & 0 & 0 \\
        0          & 0 & 0          & 0 & 0 & 0 \\
        -b_{2,5} & -b_{3,5} & 0 & 0 & 0 & b_{2,5} + b_{3,5} \\
    \end{pmatrix}\\
    &= \frac{1}{c}\begin{pmatrix}
        b_{2,5} &  b_{2,5} \\
        b_{3,5} &  b_{3,5} \\
        0 &  0 \\
       -c &  0 \\
        0 &  0 \\
        0 & -c
    \end{pmatrix} \cdot \begin{pmatrix}
        -b_{2,5} & -b_{3,5} & 0 & c & 0 & 0 \\
         b_{2,5} &  b_{3,5} & 0 & 0 & 0 & -c
    \end{pmatrix}
\end{align*}
where we used the short hand $c = b_{2,5} + b_{3,5}$.
We can then compute the inverse for the open topology as \begin{align*}
    \matr B_o^{-1} = ( \matr B_c +  \Delta \matr B)^{-1}
\end{align*}
using again the Woodbury matrix identity.


\section{Detecting islanding}

A line outage or busbar split may split up a grid into mutually disconnected islands. Here we briefly discuss how islanding can be detected from the (inverse) grounded Laplacian. This is especially important in topology optimization, where a large number of possible busbar splits are considered. In this case we need a simple criterion to exclude busbar splits that lead to islanding.

Islanding can be detected from the Laplacian as follows. One can show that the number of zero eigenvalues of the Laplacian matrix exactly equals the number of connected components of a graph~\cite{west2001introduction}. Grounding the Laplacian removes one of the zero eigenvalues such that we obtain the following statement: A grid is disconnected if and only if the grounded Laplacian has a zero eigenvalue. This condition is equivalent to a vanishing determinant, $\det(\matr B) = 0$. 

We first consider a line outage described by a change in the grounded Laplacian
\begin{align*}
    \matr B_r &\rightarrow \matr B_m  = \matr B_r - b_e \vec \nu_e \vec \nu_e^\top .
\end{align*}
We assume that the grid is connected before the outage such that $\det(\matr B_r) \neq 0$. Using the matrix determinant lemma, we can write the determinant after the outage as
\begin{align*}
    \det(\matr B_m) =  \left( 1 - b_e \vec \nu_e^\top \matr B_r^{-1} 
    \vec \nu_e  \right)  \det(\matr B_r).
\end{align*}
Hence we find the following criterion: The failure of a line $e$ leads to an islanding if 
\begin{align}
    1 - b_e \vec \nu_e^\top \matr B_r^{-1} \vec \nu_e = 0.
\end{align}

This criterion can be generalized to the opening of a busbar coupler or switch $s$. Using the notation of Sec.~\ref{sec:bsdf} and Eq.~\eqref{eq:techlem2}, we have
\begin{align*}
    \det(\matr B_o) &=     
    \left( 1 - b_s \vec \nu_s^\top \matr B_c^{-1} 
    \vec \nu_s  \right)  \det(\matr B_c) \\
    &=   \vec \nu_s^\top \big[ \matr B_o- \matr B_o \matr B_c^{-1} \matr B_o \big]  \vec \nu_s 
    \frac{1}{(\vec \nu_s^\top \vec \nu_s)^2} \; 
    \frac{\det(\matr B_c)}{b_s}
\end{align*}
The factor ${\det(\matr B_c)}/{b_s}$ remains finite in the limit $b_s  \rightarrow \infty$. It is non-zero if the original grid topology is connected. Hence we find that the opening of a busbar coupler or switch leads to an islanding if and only if
\begin{align*}
  \vec \nu_s^\top \big[ \matr B_o- \matr B_o \matr B_c^{-1} \matr B_o \big]  \vec \nu_s = 0.
\end{align*}
We note that this criterion can be extended to the case of multiple outages or multiple busbar splits as treated in the next chapter, but becomes highly sensitive to numerical errors.


\section{Multiple branch modifications}
\label{sec:multiple}

In this section we generalize the above results to multiple grid modification, including multiple line outages and multiple switches or busbar couplers. In the latter case we will introduce a convenient notation to evaluate different settings of switches or couplers.

\subsection{Finite topology modifications and line outages}

We start with the case of finite grid modifications or line outages. The resulting expressions are commonly referred to as 
generalized line outage distribution factors (GLODF) or multiple line distribution factors (MODF). 

Let $\EL = ( e_1, e_2 , \ldots, e_M )$ be a list of branches that are modified and $\Delta b_i$ the change in the susceptance of branch $e_i$. The change of the diagonal matrix $\Bd$ and the grounded Laplacian matrix can thus be written as
\begin{align*}
    &\matr \Bd_r \rightarrow \matr \Bd_m = \matr \Bd_r + \sum_i \Delta b_i \, \vec u_{e_i} \vec u_{e_i}^\top. \\
    &\matr B_r \rightarrow \matr B_m  = \matr B_r + \sum_i \Delta  b_i \, \vec \nu_{e_i} \vec \nu_{e_i}^\top  \\
    &= \matr B_r + \underbrace{\begin{pmatrix} \vec \nu_{e_1} & \ldots & \vec \nu_{e_M} \end{pmatrix}}_{=: \matr U}
    \underbrace{\begin{pmatrix} \Delta b_{1} & & 0 \\ & \ddots & \\ 0 & & \Delta b_{M} \end{pmatrix}}_{=: \matr A}
    \begin{pmatrix} \vec \nu_{e_1}^\top \\ \ldots \\ \vec \nu_{e_M}^\top \end{pmatrix} .
\end{align*}
The inverse of the grounded Laplacian in the modified grid can be computed using the Woodbury matrix identity 
\begin{align}
    \matr B_m^{-1} &= (\matr B_r+ \matr U \matr A \matr U^\top )^{-1}  
    \nonumber \\
    &= \matr B_r^{-1}   -  \matr B_r^{-1}
    \matr U \left( \matr A^{-1} + \matr U^\top \matr B_r^{-1} \matr U \right)^{-1}  \matr U^\top  \matr B_r^{-1} .
    \label{eq:sm-multiple-finite}
\end{align}
The PTDFs of the modified grid then read
\begin{align}
    PTDF^{(m)}_{a,i} = (b_a + \Delta b_a ) \vec \nu_a^\top \matr B_m^{-1} \vec u_i.
\end{align}
We remark that this formulae involves the inverse of an $M \times M$ matrix. In most applications we have $M \ll N_e$ such that this update is much more efficient than computing $\matr B_m^{-1}$ directly. In some cases, such as $(n-2)$ contingency analysis, one can even carry out the inverse analytically~\cite{kaiser2020collective}.

\subsection{Bus Merge Distribution Factors}

We now generalize the above results to describe the closing of switches or busbar couplers. As before, we summarize these flexible elements in a list  
$\EL = ( e_1, e_2 , \ldots, e_M )$.  An ideal switch $e_i$ is is treated as a branch with two possible values of the susceptance: $b_i = 0$ in the open state and $b_i \rightarrow \infty$ in the closed state. 
Starting from the results on finite topology changes in Eq.~\eqref{eq:sm-multiple-finite}, we have to evaluate the limit $b_i \rightarrow \infty$. Furthermore, we will introduce a convenient notation to treat different settings of the switches or busbar couplers. 

To proceed, we assume that the network is always connected, even if all flexible elements are open or out of operation. That is, we compute $\matr B_r$ for the grid where $\Delta b_i = 0$ for all $e_i \in \EL$. Furthermore, we define the matrix $\matr K = \matr U^\top \matr B_r^{-1} \matr U$ and its diagonal part $\matr K_d = \mbox{diag}(\matr K)$ as a convenient short-hand. We then define the variables 
\begin{align}
    \xi_i = \frac{b_i \vec \nu_{e_i}^\top \matr B_r^{-1} \vec \nu_{e_i}}{1 + b_i \vec \nu_{e_i}^\top \matr B_r^{-1} \vec \nu_{e_i} } \, .
    \label{eq:xi-a}
\end{align}
and summarize them in in the diagonal matrix
\begin{align*}
    \matr \Xi &=  \mbox{diag} \left(\xi_{1} , \ldots , \xi_{M} \right).
\end{align*}
At this point we can easily carry out the limit $b_i \rightarrow \infty$ if the switch $e_i$ is closed and obtain
\begin{align*}
    \xi_i = \left\{ \begin{array}{l l}
        0 \; & \mbox{if switch $e_i$ is open}, \\
       +1 \; & \mbox{if switch $e_i$ is closed}.
    \end{array} \right.
\end{align*}
Hence, the matrix $\matr \Xi$ is always finite with zeros or ones on the diagonal. 

We can now proceed with Eq.~\eqref{eq:sm-multiple-finite}, aiming to replace the matrix $\matr A$ with the more well-behaved matrix $\matr \Xi$. To this end, we rewrite Eq.~\eqref{eq:xi-a} as
\begin{align}
    \matr \Xi &= \matr A  \matr K_d  (\eye + \matr A \matr K_d)^{-1} \nonumber \\
    \Leftrightarrow \qquad 
    \matr A &=  \matr K_d^{-1} \matr \Xi  (\eye - \matr \Xi)^{-1} \, .
    \label{eq:A-Xi}
\end{align}
Equation~\eqref{eq:sm-multiple-finite} can thus be written as
\begin{align}
    & \matr B_m^{-1} = \matr B_r^{-1} - \matr B_r^{-1} \matr U \left( \matr A^{-1} + \matr K \right)^{-1} \matr U^\top \matr B_r^{-1} \nonumber \\
    &\quad = \matr B_r^{-1} - \matr B_r^{-1} \matr U \left( \matr K_d \matr \Xi^{-1} + \matr K - \matr K_d \right)^{-1} \matr U^\top \matr B_r^{-1} \nonumber \\
    &\quad = \matr B_r^{-1} - \matr B_r^{-1} \matr U \matr \Xi \left( \matr K_d +  ( \matr K - \matr K_d) \matr \Xi \right)^{-1} \matr U^\top \matr B_r^{-1}  .
    \label{eq:switch-zeta}
\end{align}
This expression contains only the matrix $\matr \Xi$ and not the matrix $\matr A$ such that we can easily carry out the limit $b_i \rightarrow \infty$ if the switch $e_i$ is closed. Using Eq.~\eqref{eq:switch-zeta}, we can directly read of the PTDFs in the modified grid for any branch $a \notin \EL$,
\begin{align}
    PTDF^{(m)}_{a,i} = b_a \vec \nu_a^\top \matr B_m^{-1} \vec u_i.
\end{align}


\subsection{Bus Split Distribution Factors}

We will now generalize the bus split distribution factors introduced in Sec.~\ref{sec:bsdf} to the case of multiple switches or busbar couplers. As before, the matrix $\matr B_m$ describes the grid where all switches or busbar couplers are closed and have been merged into a single node each and $\matr B_c$ where all switches and busbar couplers are closed, but where the nodes have \emph{not} been merged. We obtain $\matr B_c^{-1}$ from $\matr B_m^{-1}$ by copying the respective rows and columns. 

We now compute the impact of opening the switches or busbars summarized in the list $\EL = ( e_1, e_2 , \ldots, e_M )$. The grounded Laplacian matrix of the open grid is given by
\begin{align*}
   \matr B_o
    &= \matr B_c + \underbrace{\begin{pmatrix} \vec \nu_{e_1} & \cdots & \vec \nu_{e_M} \end{pmatrix}}_{=: \matr U}
    \underbrace{\begin{pmatrix} - b_{1} & & 0 \\ & \ddots & \\ 0 & & - b_M \end{pmatrix}}_{=: \matr A}
    \begin{pmatrix} \vec \nu_{e_1}^\top \\ \vdots \\ \vec \nu_{e_M}^\top \end{pmatrix} .
\end{align*}
with $b_i \rightarrow \infty$ for an ideal switch or busbar coupler. The Woodbury matrix identity yields
\begin{align}
    \label{eq:bsdf2-woodbury}
    & \matr B_o^{-1} = (\matr B_c+ \matr U \matr A \matr U^\top )^{-1}  
    \nonumber \\
    &= \matr B_c^{-1}   -  \matr B_c^{-1}
    \matr U \left( \matr A^{-1} + \matr U^\top \matr B_c^{-1} \matr U \right)^{-1}  \matr U^\top  \matr B_c^{-1}  
    \nonumber \\
     &= \matr B_c^{-1} -  \matr B_c^{-1}
    \matr U \matr A \left( \matr A + \matr A \matr U^\top \matr B_c^{-1} \matr U \matr A \right)^{-1} \matr A \matr U^\top  \matr B_c^{-1} .
\end{align}

As before, we need to rewrite the above expression to evaluate the limit $b_i \rightarrow \infty$. The two relations \eqref{eq:techlem1} and \eqref{eq:techlem2} are generalized to
\begin{align*}
    &\matr A \matr U^\top \matr B_c^{-1} =
    (\matr U^\top \matr U)^{-1} \matr U^\top \left(\matr B_o \matr B_c^{-1} - \eye \right) \\
    &\matr A + \matr A \matr U^\top \matr B_c^{-1} \matr U \matr A \\
    & \qquad =  (\matr U^\top \matr U)^{-1} \matr U^\top 
    \left( \matr B_o \matr B_c^{-1} \matr B_o- \matr B_o \right) 
    \matr U (\matr U^\top \matr U)^{-1} \, .
\end{align*}
Substituting these relations into Eq.~\eqref{eq:bsdf2-woodbury} yields
\begin{align}
    & \matr B_o^{-1} = \matr B_c^{-1}  +
    \left(\eye - \matr B_c^{-1} \matr B_o \right) \matr U  
    \\
    & \; \times \left[ \matr U^\top (\matr B_o - \matr B_o \matr B_c^{-1} \matr B_o) \matr U  \right]^{-1}
    \matr U^\top \left(\eye - \matr B_o \matr B_c^{-1} \right).
    \nonumber
\end{align}
We note that a substation can consist of several physical breakers or disconnectors. In this case we do not have to model each element separately as a switch, reducing the size of the matrix to be inverted. For example, any bus split can be included by modeling only the busbar coupler explicitly~\cite{van2023bus}.

Finally, given the inverse of the grounded Laplacian, one can readily compute the power transfer and line outage distribution factors.

\section{Summary and Outlook}

This article introduces a unified algebraic approach to distribution factors in linear power flow. Distribution factors are common tools in power system analysis and operation and a unified treatment can improve the applicability of these tools. The algebraic formulation facilitates the generalization of distribution factors to complex topology modifications, such as multiple line outages or bus splits. The common theme is the formulation of topology modifications as low-rank updates and the application of the Woodbury matrix identity. Similar sparsity oriented techniques have been discussed in~\cite{alsac1983sparsity}.

Our results are particularly useful in topology optimization~\cite{subramanian2021exploring}, when multiple grid topologies have to be evaluated. We show how to efficiently update the line outage distribution factors after branch modifications, bis splits and bus mergers, thus enabling an efficient $(n-1)$ security analysis. We have also shown how to efficiently detect whether a bus split leads to islanding. In the future, our algebraic treatment may be applied to generalized linear approximations~\cite{Zhang2013,Yang2019} to treat aspects of reactive power and voltage stability.

\acknowledgments 

We thank C.~Hartmann and M.~Titz for helpful comments on the manuscript. D.W. gratefully acknowledges support by the German federal ministry of research and education (Bundeministerium für Bildung und Forschung, BMBF) via the grant number 03SF0751.


%

\end{document}